\documentclass[12pt]{amsart}

\usepackage{algorithmic}
\usepackage[section]{algorithm}

\usepackage{fullpage}

\usepackage{amsfonts,amssymb,amscd,amsmath,enumerate,verbatim,color,xcolor}
\usepackage[latin1]{inputenc}
\usepackage{amscd}
\usepackage{latexsym}
\usepackage{tikz}
\usepackage{graphicx} 
\usepackage{here}
\usepackage{mathptmx}

\usepackage{hyperref}
\hypersetup{colorlinks,linkcolor=blue ,citecolor=blue, urlcolor=blue}
\def\NZQ{\mathbb}               

\def\ZZ{{\NZQ Z}}
\def\RR{{\NZQ R}}

\def\FF{{\NZQ F}}

\def\frk{\mathfrak}               

\def\Phi{{\frk N}}
\def\ab{{\mathbf a}}

\def\cb{{\mathbf c}}
\def\eb{{\mathbf e}}

\def\vb{{\mathbf v}}
\def\ub{{\mathbf u}}
\def\wb{{\mathbf w}}
\def\xb{{\mathbf x}}
\def\yb{{\mathbf y}}

\def\pb{{\mathbf p}}
\def\mb{{\mathbf m}}

\def\opn#1#2{\def#1{\operatorname{#2}}} 
\opn\gr{gr}

\def\Pc{{\mathcal P}}
\def\Qc{{\mathcal Q}}

\def\wt{{\rm wt}}
\def\hei{{\rm ht}}

\newtheorem{Theorem}{Theorem}[section]
\newtheorem{Lemma}[Theorem]{Lemma}
\newtheorem{Corollary}[Theorem]{Corollary}
\newtheorem{Proposition}[Theorem]{Proposition}

\theoremstyle{definition}

\newtheorem{Example}[Theorem]{Example}

\let\epsilon\varepsilon
\let\phi=\varphi
\let\kappa=\varkappa
\opn\dis{dis}
\opn\height{height}
\opn\dist{dist}
\def\pnt{{\raise0.5mm\hbox{\large\bf.}}}

\opn\Lex{Lex}
\opn\conv{conv}

\def\codeg{{\rm codeg}}
\numberwithin{equation}{section}

\title{Gorenstein Simplices and Even Binary Self-Complementary Codes}
\author{Akiyoshi Tsuchiya}

\address{Akiyoshi Tsuchiya,
Department of Information Science,
Faculty of Science,
Toho University,
2-2-1 Miyama, Funabashi, Chiba 274-8510, Japan} 
\email{akiyoshi@is.sci.toho-u.ac.jp}

\begin{document}

\begin{abstract}
It is known that if a Gorenstein simplex of dimension \(d\) and degree \(s\) is not a lattice pyramid, then \(d \leq 2s-1\).
In this paper, we study the extremal case \(d=2s-1\).
More precisely, we characterize Gorenstein simplices of dimension \(2s-1\) and degree \(s\) which are not lattice pyramids in terms of even binary self-complementary codes.
As an application, combining this characterization with existing classification results on reflexive simplices, we classify Gorenstein simplices of degree \(3\) and \(4\).
Equivalently, we classify polarized \(d\)-dimensional Gorenstein fake weighted projective spaces \((X,L)\) satisfying $-K_X=(d-2)L$ or $-K_X=(d-3)L$,
where \(-K_X\) is the anticanonical divisor of \(X\) and \(L\) is a Cartier divisor on \(X\).
\end{abstract}

\maketitle
\section{Introduction}

A \textit{lattice polytope} is a convex polytope all of whose vertices have integer coordinates.
Two lattice polytopes \(\Pc,\Qc\subset \RR^d\) are said to be \textit{unimodularly equivalent} if there exist a unimodular matrix \(U\in \ZZ^{d\times d}\) and a lattice point \(\wb\in \ZZ^d\) such that
\[
\Qc=f_U(\Pc)+\wb,
\]
where \(f_U:\RR^d\to \RR^d\) is the linear map defined by \(f_U(\xb)=\xb U\).
In the study of lattice polytopes, it is standard to consider classification problems up to unimodular equivalence.
Moreover, a lattice simplex \(\Delta\) can be described in terms of its associated finite abelian group \(\Lambda_\Delta\), and unimodular equivalence classes of lattice simplices can be recovered from these groups.
Using associated finite abelian groups and coding-theoretic methods, Batyrev and Hofscheier \cite{BatyrevHof} classified lattice polytopes whose \(h^*\)-polynomials are binomials.
This point of view will play an important role in the present paper, and we recall the precise correspondence in Section~\ref{sec:preliminaries}.

Let \(\Pc \subset \RR^d\) be a lattice polytope of dimension \(d\).
The \textit{lattice pyramid} over \(\Pc\) is the lattice polytope
\[
{\rm Pyr}(\Pc):=\conv(\Pc\times\{0\},(0,\dots,0,1))\subset \RR^{d+1}.
\]
More generally, we say that a lattice polytope is a lattice pyramid if it is obtained by successively taking lattice pyramids over a lower-dimensional lattice polytope.
Since lattice pyramid constructions preserve several important properties and invariants, it is natural to study classification problems for lattice polytopes which are not lattice pyramids.
For this reason, many classification problems in the theory of lattice polytopes are naturally formulated under the assumption that the polytope is not a lattice pyramid.

The \textit{codegree} of \(\Pc\), denoted by \(\codeg(\Pc)\), is defined by
\[
\codeg(\Pc)=\min\{\,n\in \ZZ_{>0} : {\rm int}(n\Pc)\cap \ZZ^d\neq \emptyset\,\},
\]
where \(n\Pc:=\{n\xb:\xb\in \Pc\}\) and \({\rm int}(\Pc)\) denotes the interior of \(\Pc\).
The \textit{degree} of \(\Pc\), denoted by \(\deg(\Pc)\), is defined by
\[
\deg(\Pc)=d+1-\codeg(\Pc).
\]
These invariants play important roles in the study of lattice polytopes and their Ehrhart theory.
We recall in Section~\ref{sec:preliminaries} that \(\deg(\Pc)\) coincides with the degree of the \(h^*\)-polynomial of \(\Pc\).

If one restricts attention to lattice polytopes which are not lattice pyramids, then the degree imposes strong restrictions on the dimension.
More precisely, it was shown in \cite{Nillbound} that if a lattice polytope of dimension \(d\) and degree \(s\) has \(d+c+1\) vertices and is not a lattice pyramid, then
\[
d \leq c(2s+1)+4s-2.
\]
In particular, for lattice simplices one has
\[
d \leq 4s-2.
\]
This naturally leads to the extremal case
\[
d=4s-2
\]
for lattice simplices.
Higashitani \cite{HigashitaniGor} characterized lattice simplices in this extremal case and showed that they arise from binary simplex codes.
Thus binary codes already appear naturally in the study of lattice simplices at the boundary of the dimension-degree inequality.

On the other hand, Gorenstein polytopes form one of the most important classes of lattice polytopes.
A lattice polytope \(\Pc\subset \RR^d\) is called \textit{reflexive} if the origin of \(\RR^d\) belongs to the interior of \(\Pc\) and the dual polytope of \(\Pc\) is again a lattice polytope.
A lattice polytope \(\Pc\) is called \textit{Gorenstein of index \(r\)} if \(r\Pc\) is unimodularly equivalent to a reflexive polytope.
If \(\Pc\) is Gorenstein of index \(r\), then \(\codeg(\Pc)=r\).
Moreover, Gorenstein polytopes are characterized by the palindromicity of their \(h^*\)-polynomials.
They arise naturally in several areas such as toric geometry, commutative algebra, and mirror symmetry, and in each fixed dimension there exist only finitely many Gorenstein polytopes up to unimodular equivalence.
For these reasons, the classification of Gorenstein polytopes has been regarded as a fundamental problem.
For example, several classification results on Gorenstein simplices based on associated finite abelian groups were obtained in \cite{HibiTsuchiyaYoshida,HigashitaniNillTsuchiya,TsuchiyaGorsimp}.

For Gorenstein polytopes, a stronger bound is known: if a Gorenstein polytope of dimension \(d\) and degree \(s\) is not a lattice pyramid, then \(d\le 3s-1\) \cite{Nilldegreebound}.
If one restricts further to Gorenstein simplices, then a much stronger bound is available.
In fact, it was shown in \cite{polyadj} that if a Gorenstein simplex of dimension \(d\) and degree \(s\) is not a lattice pyramid, then
\[
d \leq 2s-1.
\]
Thus, for Gorenstein simplices, the extremal case is
\[
d=2s-1.
\]
The aim of this paper is to study this extremal case.
More precisely, we investigate Gorenstein simplices of dimension \(2s-1\) and degree \(s\) which are not lattice pyramids.
Although this situation is much more restrictive than Higashitani's extremal case for arbitrary lattice simplices, binary codes still govern the classification.
Our first main result shows that such simplices are characterized by even binary self-complementary codes.
Here a binary self-complementary code means a binary linear code containing the all-one vector.

\begin{Theorem}\label{thm:intro-main}
Let \(\Delta\) be a \((2s-1)\)-dimensional Gorenstein simplex of degree \(s\) which is not a lattice pyramid.
Then there exists an even binary self-complementary code \(C \subset \mathbb{F}_2^{2s}\) such that the associated finite abelian group of \(\Delta\) is of the form
\[
\Lambda_\Delta=
\left\{
\frac{1}{2}\cb \in [0,1)^{2s} : \cb \in C
\right\}.
\]
Conversely, every even binary self-complementary code of length \(2s\) arises from a \((2s-1)\)-dimensional Gorenstein simplex of degree \(s\) which is not a lattice pyramid.
\end{Theorem}

Thus, in the extremal case \(d=2s-1\), the classification of Gorenstein simplices is reduced to the classification of even binary self-complementary codes.
Moreover, the corresponding \(h^*\)-polynomials are determined by the weight distributions of these codes (Corollary \ref{cor:weight}).

Batyrev and Juny \cite{BatyrevJuny} classified Gorenstein toric Del Pezzo varieties in arbitrary dimension; equivalently, they classified Gorenstein polytopes of degree \(2\).
On the other hand, for Gorenstein polytopes of degree at least \(3\), classification problems become much more difficult in general.
As applications of Theorem~\ref{thm:intro-main}, we obtain classifications of Gorenstein simplices of degrees \(3\) and \(4\).
Equivalently, we classify polarized \(d\)-dimensional Gorenstein fake weighted projective spaces \((X,L)\) satisfying
\[
-K_X=(d-2)L \quad \text{or} \quad -K_X=(d-3)L,
\]
where \(-K_X\) is the anticanonical divisor of \(X\) and \(L\) is a Cartier divisor on \(X\).

Classical classification results are available for reflexive polytopes in low dimensions, and hence also for reflexive simplices.
In particular, Kreuzer and Skarke \cite{reflexive3,reflexive4} showed that there exist exactly \(4{,}319\) unimodular equivalence classes of reflexive polytopes in dimension \(3\) and exactly \(473{,}800{,}776\) unimodular equivalence classes of reflexive polytopes in dimension \(4\).
On the other hand, if a Gorenstein simplex of degree \(3\) is not a lattice pyramid, then its dimension is at most \(5\).
Hence only the cases of dimensions \(3\), \(4\), and \(5\) can occur.
In dimension \(3\), they are precisely reflexive simplices, and there exist exactly \(48\) unimodular equivalence classes.
In dimension \(4\), they are precisely Gorenstein simplices of index \(2\) which are not lattice pyramids, and there exist exactly \(13\) unimodular equivalence classes.
Therefore the only remaining case is dimension \(5\), which is precisely the extremal case covered by Theorem~\ref{thm:intro-main}.
In this way, we obtain the following explicit classification.

\begin{Theorem}\label{thm:class3_intro}
There exist exactly \(6\) unimodular equivalence classes of \(5\)-dimensional Gorenstein simplices of degree \(3\) which are not lattice pyramids.
\end{Theorem}
We record the result here and refer to its full form in Theorem~\ref{thm:degree3-classification}.
Combined with the known classifications in dimensions \(3\) and \(4\), this yields the complete classification of Gorenstein simplices of degree \(3\).

While the classification of reflexive polytopes in higher dimensions is still very difficult in general, Ghirlanda \cite{classrefsim} recently classified reflexive simplices up to dimension \(6\); see also \cite{GhirlandaData}.
On the other hand, if a Gorenstein simplex of degree \(4\) is not a lattice pyramid, then its dimension is at most \(7\).
Hence only the cases of dimensions \(4\), \(5\), \(6\), and \(7\) can occur.
In dimension \(4\), they are precisely reflexive simplices, and there exist exactly \(1{,}561\) unimodular equivalence classes.
In dimension \(5\), they are precisely Gorenstein simplices of index \(2\) which are not lattice pyramids, and there exist exactly \(264\) unimodular equivalence classes.
In dimension \(6\), they are precisely Gorenstein simplices of index \(3\) which are not lattice pyramids, and there exist exactly \(47\) unimodular equivalence classes.
Therefore the only remaining case is dimension \(7\), which is again the extremal case covered by Theorem~\ref{thm:intro-main}.
This yields the following classification.

\begin{Theorem}\label{thm:class4_intro}
There exist exactly \(19\) unimodular equivalence classes of \(7\)-dimensional Gorenstein simplices of degree \(4\) which are not lattice pyramids.
\end{Theorem}
We record the result here and refer to its full form in Theorem~\ref{thm:degree4-classification}.
Combined with the known classifications in dimensions \(4\), \(5\), and \(6\), this yields the complete classification of Gorenstein simplices of degree \(4\).

Once Theorem~\ref{thm:intro-main} is established, the classifications in Theorems~\ref{thm:class3_intro} and \ref{thm:class4_intro}, as well as analogous classifications for larger values of \(s\), can in principle be obtained by computer search. In this paper, however, we give theoretical proofs for the cases \(s=3\) and \(4\).

The paper is organized as follows.
In Section~\ref{sec:preliminaries}, we recall basic facts on \(h^*\)-polynomials, Gorenstein polytopes, finite abelian groups associated to lattice simplices, and binary linear codes.
In Section~\ref{sec:main}, we prove the classification of \((2s-1)\)-dimensional Gorenstein simplices of degree \(s\) in terms of even binary self-complementary codes.
In Section~\ref{sec:degree3}, we apply this result to classify Gorenstein simplices of degree \(3\) and \(4\).

\subsection*{Acknowledgment}
The author is grateful to Marco Ghirlanda for helpful correspondence and for providing numerical data on non-pyramidal Gorenstein simplices in small dimensions.
This work was supported by JSPS KAKENHI 22K13890 and 26K00618.

\section{Preliminaries}\label{sec:preliminaries}

In this section, we recall basic facts on \(h^*\)-polynomials, Gorenstein polytopes, finite abelian groups associated to lattice simplices, and binary linear codes.

\subsection{Ehrhart theory and Gorenstein polytopes}
Let $\Pc \subset \RR^d$ be a lattice polytope of dimension $d$.
Given a positive integer $k$, we define
$$L_{\Pc}(k)=|k \Pc \cap \ZZ^d|.$$
The study on $L_{\Pc}(k)$ originated in Ehrhart \cite{Ehrhart} who proved that $L_{\Pc}(k)$ is a polynomial in $k$ of degree $d$ with the constant term $1$.
We call $L_{\Pc}(k)$ the \textit{Ehrhart polynomial} of $\Pc$.
The generating function of the lattice point enumerator, i.e., the formal power series
$$\text{Ehr}_\Pc(t)=1+\sum\limits_{k=1}^{\infty}L_{\Pc}(k)t^k$$
is called the \textit{Ehrhart series} of $\Pc$.
It is known that it can be expressed as a rational function of the form
$$\text{Ehr}_\Pc(t)=\frac{h^*(\Pc,t)}{(1-t)^{d+1}},$$
where $h^*(\Pc,t)$ is a polynomial in $t$ of degree at most $d$ with nonnegative integer coefficients \cite{Stanley_nonnegative} and it
is called
the \textit{$h^*$-polynomial} of $\Pc$. 
Moreover, $$h^*(\Pc,t)=\sum_{i=0}^{d} h_i^* t^i$$ satisfies $h^*_0=1$, $h^*_1=|\Pc \cap \ZZ^d|-(d +1)$ and $h^*_{d}=|{\rm int} (\Pc) \cap \ZZ^d|$.
Furthermore, $h^*(\Pc,1)=\sum_{i=0}^{d} h_i^*$ is equal to the normalized volume of $\Pc$.
Note that $\deg(\Pc)$ coincides with the degree of $h^*(\Pc,t)$.
It is well-known that $h^*(\Pc,t)=h^*({\rm Pyr}(\Pc),t)$.
We refer the reader to \cite{BeckRobins2015book} for the detailed information about Ehrhart polynomials and $h^*$-polynomials.

A lattice polytope $\Pc \subset \RR^d$ of dimension $d$ is called \textit{reflexive} if the origin of $\RR^d$ belongs to the interior of $\Pc$ and its dual polytope 
\[\Pc^\vee:=\{\yb \in \RR^d  :  \langle \xb,\yb \rangle \leq 1 \ \text{for all}\  \xb \in \Pc \}\]
is also a lattice polytope, where $\langle \xb,\yb \rangle$ is the usual inner product of $\RR^d$.
A lattice polytope \(\Pc\) is called \textit{Gorenstein of index \(r\)} if \(r\Pc\) is unimodularly equivalent to a reflexive polytope.
We can characterize Gorenstein polytopes by their $h^*$-polynomials.
\begin{Lemma}[\cite{DeNegriHibi}]
Let $\Pc$ be a lattice polytope of degree $s$, and write
\[
h^*(\Pc,t)=h_0^*+h_1^*t+\cdots+h_s^*t^s.
\]
Then $\Pc$ is Gorenstein if and only if $h^*(\Pc,t)$ is palindromic, namely,
for any $0 \leq i \leq s$, one has
\[
h^*_i=h^*_{s-i}.
\]
\end{Lemma}

\subsection{Finite abelian groups associated to lattice simplices}

Let
\[
\Delta=\conv(\vb_0,\dots,\vb_d)\subset \RR^d
\]
be a lattice simplex of dimension \(d\).
We define the associated finite abelian group of \(\Delta\) by
\[
\Lambda_\Delta=
\left\{
(x_0,\dots,x_d)\in [0,1)^{d+1}:
\sum_{i=0}^d x_i(\vb_i,1)\in \ZZ^{d+1}
\right\}.
\]
The group operation is addition modulo \(1\) in each coordinate.
For
\[
\xb=(x_0,\dots,x_d)\in \Lambda_\Delta,
\]
we define the \textit{height} of \(\xb\) by
\[
\hei(\xb):=\sum_{i=0}^d x_i.
\]
Since
\[
\sum_{i=0}^d x_i(\vb_i,1)\in \ZZ^{d+1},
\]
the number \(\hei(\xb)\) is always an integer.
For \(\xb=(x_0,\dots,x_d)\in \Lambda_\Delta\), we denote by \(-\xb\) the inverse of \(\xb\) in \(\Lambda_\Delta\).
Moreover, if \(0<x_i<1\) for all \(i\), then
\[
-\xb=(1-x_0,\dots,1-x_d)
\]
and $\hei(-\xb)=d+1-\hei(\xb)$.

The following facts are well-known.
\begin{Lemma}\label{lem:hstar-lambda}
Let \(\Delta\) be a lattice simplex of dimension \(d\), and write
\[
h^*(\Delta,t)=h_0^*+h_1^*t+\cdots+h_d^*t^d.
\]
Then for any $0 \leq i \leq d$, one has
\[
h_i^*=|\{\xb\in \Lambda_\Delta:\hei(\xb)=i\}|.
\]
\end{Lemma}

\begin{Lemma}[\cite{BatyrevHof}]\label{lem:pyramid-lambda}
Let \(\Delta\) be a lattice simplex of dimension \(d\).
Then \(\Delta\) is a lattice pyramid if and only if there exists an index \(i\in \{0,\dots,d\}\) such that for all $\xb=(x_0,\dots,x_d)\in \Lambda_\Delta$,
$x_i=0$.
\end{Lemma}

We also recall the correspondence between lattice simplices and finite abelian subgroups.

\begin{Proposition}[\cite{BatyrevHof}]\label{prop:delta-lambda-correspondence}
Let \(d\ge 1\).
Then there is a bijection between unimodular equivalence classes of \(d\)-dimensional lattice simplices with an ordered set of vertices and finite abelian subgroups \(G\subset [0,1)^{d+1}\) under addition modulo \(1\) such that
for all $(x_0,\dots,x_d)\in G$, 
$\sum_{i=0}^d x_i\in \ZZ.$
Under this correspondence, two lattice simplices are unimodularly equivalent if and only if their associated groups coincide up to a permutation of coordinates.
\end{Proposition}

\subsection{Binary linear codes}

Let \(n\ge 1\).
A \textit{binary linear code} of length \(n\) is a linear subspace \(C\subset \FF_2^n\).
For
\[
\cb=(c_1,\dots,c_n)\in \FF_2^n,
\]
the \textit{Hamming weight} of \(\cb\) is defined by
\[
\wt(\cb):=|\{i:c_i\neq 0\}|.
\]
The code \(C\) is called \textit{even} if every codeword of \(C\) has even weight.

In this paper, a binary linear code \(C\subset \FF_2^n\) is called \textit{self-complementary} if it contains the all-one vector
\[
\mathbf{1}_n:=(1,\dots,1)\in \FF_2^n.
\]

For a binary linear code \(C\subset \FF_2^{2s}\), we will consider the subset
\[
\frac12 C:=\left\{\frac12\cb\in [0,1)^{2s}:\cb\in C\right\}.
\]
If \(C\) is even, then every element of \(\frac12 C\) has integral height, and hence \(\frac12 C\) satisfies the integrality condition in Proposition~\ref{prop:delta-lambda-correspondence}.

\section{Gorenstein simplices of dimension \(2s-1\) and degree \(s\)}
\label{sec:main}
In this section, we prove Theorem\ref{thm:intro-main}.
First, we show the following.

\begin{Lemma}\label{lem:lambda-one-implies-pyramid}
Let
$
\Delta=\conv(\vb_0,\dots,\vb_d)\subset \RR^d
$
be a Gorenstein simplex of dimension $d$ and codegree \(r\), and let
$
\pb=\sum_{i=0}^d \lambda_i \vb_i
$
be the unique interior lattice point of \(r\Delta\), where
$
\lambda_i>0
$ for all $0 \leq i \leq d$ and $
\sum_{i=0}^d \lambda_i=r$.
Then one has $0 < \lambda_i \leq 1$ for all $i$. Moreover,
if \(\lambda_j=1\) for some \(j\), then
\(\Delta\) is a lattice pyramid.
\end{Lemma}

\begin{proof}
Assume that $\lambda_j >1$ for some $j$.
It then follows from $\sum_{i=0}^d \lambda_i=r$, $\lambda_i > 0$ for all $i \neq j$ and $\lambda_j-1>0$ that
\[\vb:=\pb-\vb_j=\sum_{i \neq j} \lambda_i \vb_i+(\lambda_j-1)\vb_j\]
belongs to ${\rm int}((r-1)\Delta) \cap \ZZ^{d}$.
This implies that $r=\codeg(\Delta) \leq r-1$, a contradiction.
Hence one has $0< \lambda_i \leq 1$ for all $i$.

For each \(0\le i\le d\), set
\[
\ab_i:=(\vb_i,1)\in \ZZ^{d+1},
\]
and consider the cone
\[
\mathcal C_\Delta:=\RR_{\ge 0}\ab_0+\cdots+\RR_{\ge 0}\ab_d.
\]
Since \(\Delta\) is a simplex, the vectors \(\ab_0,\dots,\ab_d\) are linearly independent.
Let
\[
\cb:=(\pb,r)=\sum_{i=0}^d \lambda_i \ab_i.
\]
Since \(\Delta\) is Gorenstein of codegree \(r\), the cone \(\mathcal C_\Delta\) is Gorenstein with canonical lattice point \(\cb\). Namely, one has
\[
\operatorname{int}(\mathcal C_\Delta)\cap \ZZ^{d+1}
=\cb+\bigl(\mathcal C_\Delta\cap \ZZ^{d+1}\bigr).
\]

Assume that \(\lambda_j=1\) for some $j$.
Then
\[\ub:=\cb-\ab_j=\sum_{i\ne j}\lambda_i \ab_i
\]
belongs to \(\mathcal C_\Delta\cap \ZZ^{d+1}\).

Suppose, to the contrary, that there exists
\[
\yb=(y_0,\dots,y_d)\in \Lambda_\Delta
\]
with \(y_j>0\) for some $j$.
Set
\[
\mb_{\yb}:=\sum_{i=0}^d y_i \ab_i.
\]
Since \(\yb\in\Lambda_\Delta\), one has
$
\mb_{\yb}\in \ZZ^{d+1}.
$
Moreover, \(0\le y_i<1\) for all \(i\).
Set
\[
\ub+\mb_{\yb}
=
\sum_{i\ne j}(\lambda_i+y_i)\ab_i+y_j \ab_j.
\]
    Since \(\lambda_i + y_i \geq \lambda_i>0\) for all $i \neq j$ and \(y_j>0\), one has 
\[
\ub+\mb_{\yb}\in \operatorname{int}(\mathcal C_\Delta)\cap \ZZ^{d+1}.
\]
By the Gorenstein property of \(\mathcal C_\Delta\), there exists
\[
\mb'\in \mathcal C_\Delta\cap \ZZ^{d+1}
\]
such that
\[
\ub+\mb_{\yb}=\cb+\mb'.
\]
Since \(\ub=\cb-\ab_j\), this implies
\[
\mb_{\yb}-\ab_j=\mb'.
\]
Therefore
$\mb_{\yb}-\ab_j\in \mathcal C_\Delta$.

On the other hand,
\[
\mb_{\yb}-\ab_j
=
\sum_{i\ne j} y_i \ab_i + (y_j-1)\ab_j.
\]
Since \(0<y_j<1\), $y_j -1 <0$.
Because \(\ab_0,\dots,\ab_d\) are linearly independent, an element of \(\mathcal C_\Delta\) has a unique expression as a nonnegative linear combination of \(\ab_0,\dots,\ab_d\).
Thus \(\mb_{\yb}-\ab_j\notin \mathcal C_\Delta\), a contradiction.

Hence \(y_j=0\) for every \(\yb\in\Lambda_\Delta\).
Therefore, \(\Delta\) is a lattice pyramid.
\end{proof}

\begin{Lemma}
\label{lem:full-support-top}
Let $\Delta=\conv(\vb_0,\ldots,\vb_{d})$ be a Gorenstein simplex of dimension $d$ and degree $s$ which is not a lattice pyramid.
Let
\[
\xb=(x_0,\dots,x_{d})\in\Lambda_\Delta
\]
be the unique element of $\Lambda_\Delta$ with
$
\hei(\xb)=s$.
Then one has $x_i>0$ for all $i$.
\end{Lemma}
\begin{proof}
Set $r:=\codeg(\Delta)=d+1-s$ and let 
$
\pb=\sum_{i=0}^d \lambda_i \vb_i
$
be the unique interior lattice point of \(r\Delta\), where
$
\lambda_i>0
$ for all $0 \leq i \leq d$ and $
\sum_{i=0}^d \lambda_i=r$.
Since $\Delta$ is not a lattice pyramid,
it then follows from Lemma \ref{lem:lambda-one-implies-pyramid} that $0 < \lambda_i < 1$ for all $i$. Since $(\pb,r)=\sum_{i=0}^d \lambda_i(\vb_i,1)\in \ZZ^{d+1}$ and $0<\lambda_i<1$ for all $i$, it follows that
\[
\lambda:=(\lambda_0,\ldots,\lambda_d)\in \Lambda_\Delta
\]
with $\hei(\lambda)=r$.
Set $\yb=(y_0,\ldots,y_d):=-\lambda \in \Lambda_{\Delta}$. Then $0 <y_i <1$ for all $i$. Since $\hei(\yb)=d+1-\hei(\lambda)=d+1-r=s$, one has $\xb=\yb$. Therefore, we obtain $x_i > 0$ for all $i$.
\end{proof}

The assumption of Gorensteinness for this lemma is essential.
\begin{Example}
    Consider the finite subgroup
\[
\Lambda=\langle \xb,\yb\rangle \subset [0,1)^6
\]
generated by
\[
\xb=\left(\frac34,\frac34,\frac12,\frac12,\frac12,0\right),
\qquad
\yb=\left(0,\frac12,0,\frac12,\frac12,\frac12\right).
\]
Then \(|\Lambda|=8\), and the elements of \(\Lambda\) are
\[
\begin{aligned}
&(0,0,0,0,0,0),
\left(\frac12,\frac12,0,0,0,0\right),
\left(0,\frac12,0,\frac12,\frac12,\frac12\right),
\left(\frac14,\frac14,\frac12,\frac12,\frac12,0\right),\\
&\left(\frac14,\frac34,\frac12,0,0,\frac12\right),
\left(\frac12,0,0,\frac12,\frac12,\frac12\right),
\left(\frac34,\frac14,\frac12,0,0,\frac12\right),
\left(\frac34,\frac34,\frac12,\frac12,\frac12,0\right).
\end{aligned}
\]
Hence, by Proposition~\ref{prop:delta-lambda-correspondence}, there exists a lattice simplex \(\Delta\) such that \(\Lambda_\Delta=\Lambda\). Since no coordinate is identically zero on \(\Lambda\), Lemma~\ref{lem:pyramid-lambda} implies that \(\Delta\) is not a lattice pyramid.
Moreover, one has
\[
h^*(\Delta,t)=1+t+5t^2+t^3.
\]
In particular, the maximal height is \(3\), and the unique element of height \(3\) is
\[
\left(\frac34,\frac34,\frac12,\frac12,\frac12,0\right).
\]
Therefore, the assumption that \(\Delta\) is Gorenstein is essential in Lemma~\ref{lem:full-support-top}.
\end{Example}

Now, we prove Theorem \ref{thm:intro-main}.
\begin{proof}[Proof of Theorem\ref{thm:intro-main}]
For each \(0\le k\le s\), let
\[
H_k:=\{\yb\in\Lambda_\Delta:\hei(\yb)=k\}.
\]
Since \(\Delta\) is Gorenstein of degree \(s\), the \(h^*\)-polynomial of \(\Delta\) is palindromic of degree \(s\). Hence for any $0 \leq k \leq s$, one has
$|H_k|=|H_{s-k}|$.
Moreover, since the coefficient of \(t^s\) in \(h^*(\Delta,t)\) is equal to \(1\), there exists a unique element
$\xb\in\Lambda_\Delta$
such that
$\hei(\xb)=s$.
Since $\Delta$ is not a lattice pyramid, by Lemma~\ref{lem:full-support-top}, one has
$0<x_i<1$ for all $i$.
Therefore,
$\hei(-\xb)=2s-\hei(\xb)=s$.
Since $\xb$ is the unique element of height $s$, we obtain
$-\xb=\xb$.
Thus $2\xb=\mathbf{0}$.
Since $0<x_i<1$ for all $i$, it follows that
$x_i=\frac{1}{2}$ for all $i$.
Hence
\[
\xb=\left(\frac12,\dots,\frac12\right).
\]

Next, we claim that for any $0 \leq k \leq s$, 
\[
\xb+H_k=H_{s-k}.
\]
We prove this by induction on \(k\).
For \(k=0\), since \(H_0=\{\mathbf{0}\}\) and \(H_s=\{\xb\}\), one has
\[
\xb+H_0=H_s.
\]
Now let \(1\le k\le s\), and assume that for any $0 \leq j < k$,
\[
\xb+H_j=H_{s-j}.
\]
Take
\[
\yb=(y_0,\dots,y_{2s-1})\in H_k.
\]
Set
\[
m(\yb):=\left|\{\,i: y_i\ge 1/2\,\}\right|.
\]
It then follows from \(\hei(\yb)=k\) that 
$m(\yb)\le 2k$.
On the other hand, since
\[
\xb=\left(\frac12,\dots,\frac12\right),
\]
one has
\[
\hei(\xb+\yb)=\hei(\xb)+\hei(\yb)-m(\yb)=s+k-m(\yb).
\]
Therefore
\[
\hei(\xb+\yb)\ge s-k.
\]
Assume that
\[
\hei(\xb+\yb)=s-j
\]
for some \(j<k\).
Then
\[
\xb+\yb\in H_{s-j}=\xb+H_j
\]
by the induction hypothesis.
However, since 
\[
\xb+(\xb+H_j)=2\xb+H_j=H_j,
\]
one has
\[\yb=2\xb+\yb=\xb+(\xb+\yb) \in \xb+(\xb+H_j) = H_j,\]
which contradicts \(\hei(\yb)=k\).
Thus
since \(\hei(\xb+\yb)\ge s-k\), we obtain
\[
\hei(\xb+\yb)=s-k.
\]
Hence
\[
\xb+H_k\subseteq H_{s-k}.
\]
It then follows from \( |\xb+H_k|=|H_k|=|H_{s-k}|\) that
\[
\xb+H_k=H_{s-k}.
\]

Now let
\[
\yb=(y_0,\dots,y_{2s-1})\in H_k.
\]
Since \(\xb+\yb\in H_{s-k}\), we have
\[
s-k=\hei(\xb+\yb)=s+k-m(\yb),
\]
and hence
\[
m(\yb)=2k.
\]
It then follows from $\hei(\yb)=k$
that all of nonzero coordinates must be equal to \(1/2\) and the remaining coordinates must be equal to \(0\).
Therefore every element of \(\Lambda_\Delta\) has each coordinate equal to either \(0\) or \(1/2\).

Define
\[
C:=\{\,2\yb : \yb\in\Lambda_\Delta\,\}\subset \mathbb F_2^{2s}.
\]
Since \(\Lambda_\Delta\) is a finite abelian group under addition modulo \(1\), the set \(C\) is a binary linear code.
Moreover, if \(\yb\in H_k\), then \(\yb\) has exactly \(2k\) coordinates equal to \(1/2\), so the corresponding codeword \(2\yb\) has Hamming weight \(2k\).
Hence every codeword of \(C\) has even weight, that is, \(C\) is an even binary linear code.
Since
\[
\xb=\left(\frac12,\dots,\frac12\right)\in\Lambda_\Delta,
\]
the code \(C\) contains $\mathbf{1}_{2s}$. Hence $C$ is self-complementary.
On the other hand, by construction,
\[
\Lambda_\Delta
=
\left\{
\frac12\cb\in [0,1)^{2s}:\cb\in C
\right\}.
\]

Conversely, let \(C\subset \mathbb F_2^{2s}\) be an even binary self-complementary code, and define
\[
\Lambda:=
\left\{
\frac12\cb\in [0,1)^{2s}:\cb\in C
\right\}.
\]
Then \(\Lambda\) is a finite subgroup of \([0,1)^{2s}\).
Moreover, since $C$ is even, one has $\hei(\xb) \in \ZZ$ for each $\xb \in \Lambda$. Hence
there exists a lattice simplex \(\Delta\) of dimension $2s-1$ such that
\[
\Lambda_\Delta=\Lambda.
\]
Since  \(\mathbf{1}_{2s} \in C\), the group \(\Lambda_\Delta\) contains
\[
\left(\frac12,\dots,\frac12\right),
\]
so \(\Delta\) is not a lattice pyramid.
Moreover, one has
\[
h^*(\Delta,t)=\sum_{\cb\in C} t^{\wt(\cb)/2}.
\]
Since \(C\) contains $\mathbf{1}_{2s}$, its degree is exactly \(s\). Also, adding $\mathbf{1}_{2s}$ gives a bijection between codewords of weight \(2k\) and codewords of weight \(2s-2k\), so \(h^*(\Delta,t)\) is palindromic. Therefore \(\Delta\) is Gorenstein of degree \(s\).
This completes the proof.
\end{proof}

From the proof of Theorem \ref{thm:intro-main} we obtain the following corollary.
\begin{Corollary}\label{cor:weight}
Let \(\Delta\) be a \((2s-1)\)-dimensional Gorenstein simplex of degree \(s\) which is not a lattice pyramid.
Let \(C \subset \mathbb{F}_2^{2s}\) be an even binary self-complementary code such that the associated finite abelian group of \(\Delta\) is of the form
\[
\Lambda_\Delta=
\left\{
\frac{1}{2}\cb \in [0,1)^{2s} : \cb \in C
\right\}.
\]
Then one has
\[
h^*(\Delta,t)=\sum_{\cb\in C} t^{\wt(\cb)/2}.
\]
\end{Corollary}

\section{Classification of Gorenstein simplices of degree \(3\) and $4$}\label{sec:degree3}

In this section, we classify Gorenstein simplices of degree \(3\) and $4$ in the extremal case.
We first observe a general fact: the associated finite abelian group is determined by the lower half of the height decomposition.
Let $\Delta$ be a lattice simplex of dimension $d$, and 
for each \(0\le i\le s\), let
\[
H_i:=\{\yb\in \Lambda_\Delta:\hei(\yb)=i\}.
\]
Then from the proof of Theorem \ref{thm:intro-main} we obtain the following.
\begin{Proposition}\label{prop:H_i-determine}
Let \(\Delta\) be a \((2s-1)\)-dimensional Gorenstein simplex of degree \(s\) which is not a lattice pyramid.
Then \(\Lambda_\Delta\) is completely determined by
\[
H_1,\dots,H_{\lfloor s/2\rfloor}.
\]
\end{Proposition}

\begin{proof}
This follows since for each \(0\le i\le s\) one has
\[
\xb+H_i=H_{s-i},
\]
where \[
\xb=\left(\frac12,\dots,\frac12\right).
\]
and $H_0=\{\mathbf{0}\}$.
\end{proof}

It is useful to describe $H_1$ by using a graph.
For each \(1\le i<j\le 2s\), let
\[
\varepsilon_{ij}\in [0,1)^{2s}
\]
be the element whose \(i\)-th and \(j\)-th coordinates are equal to \(1/2\), and whose other coordinates are \(0\).
Then every element of \(H_1\) is of the form \(\varepsilon_{ij}\).
Hence \(H_1\) can be regarded as the edge set of a graph on \(\{1,\ldots,2s\}\).
Define a graph \(G_\Delta\) on the vertex set \(\{1,\dots,2s\}\) by
\[
\{i,j\}\in E(G_\Delta)
\quad\Longleftrightarrow\quad
\varepsilon_{ij}\in H_1,
\]
where $E(G_{\Delta})$ is the edge set of $G_{\Delta}$.
We first record two closure properties of \(G_\Delta\).

\begin{Lemma}\label{lem:graph-closure}
The graph \(G_\Delta\) satisfies the following properties:
\begin{enumerate}
\item
if \(\{i,j\},\{i,k\}\in E(G_\Delta)\) with \(i,j,k\) pairwise distinct, then \(\{j,k\}\in E(G_\Delta)\);

\item
if \(\{i_1,j_1\},\ldots,\{i_{s-1},j_{s-1}\}\) are disjoint edges, then
\[
\{p,q\}=\{1,\ldots,2s\}\setminus \{i_1,\ldots,i_{s-1},j_1,\ldots,j_{s-1}\}
\]
is an edge of \(G_\Delta\).
In particular, if \(G_\Delta\) has an \((s-1)\)-matching, then it has an \(s\)-matching.
\end{enumerate}
\end{Lemma}

\begin{proof}
(1) Since \(\varepsilon_{ij},\varepsilon_{ik}\in H_1\subset \Lambda_\Delta\), one has
\[
\varepsilon_{ij}+\varepsilon_{ik}=\varepsilon_{jk}\in \Lambda_\Delta.
\]
Moreover,
\[
\hei(\varepsilon_{jk})=1,
\]
so \(\varepsilon_{jk}\in H_1\).

(2) Since \(\varepsilon_{i_1j_1},\ldots,\varepsilon_{i_{s-1}j_{s-1}}\in H_1\subset \Lambda_\Delta\), one has
\[
\xb+\varepsilon_{i_1j_1}+\cdots+\varepsilon_{i_{s-1}j_{s-1}}=\varepsilon_{pq}\in \Lambda_\Delta,
\]
where
\[
\{p,q\}=\{1,\ldots,2s\}\setminus \{i_1,\ldots,i_{s-1},j_1,\ldots,j_{s-1}\}.
\]
Since \(\hei(\varepsilon_{pq})=1\), one has \(\varepsilon_{pq}\in H_1\).
\end{proof}
\subsection{Classification of the $5$-dimensional Gorenstein simplices of degree $3$}
From now on, let \(\Delta\) be a \(5\)-dimensional Gorenstein simplex of degree \(3\) which is not a lattice pyramid.
Then \(\Lambda_\Delta\) is completely determined by \(H_1\).

\begin{Lemma}\label{lem:graph-classification}
The graph \(G_\Delta\) is a disjoint union of complete graphs.
Moreover, up to graph isomorphism, \(G_\Delta\) is one of the following:
\[
\emptyset,
K_2,
K_3,
3K_2,
K_4\sqcup K_2,
K_6,
\]
where $K_n$ is a complete graph of $n$ vertices.
\end{Lemma}

\begin{proof}
By Lemma~\ref{lem:graph-closure} (1), every connected component of \(G_{\Delta}\) is a complete graph.
Hence \(G_{\Delta}\) is a disjoint union of complete graphs.
Thus, up to graph isomorphism, \(G_{\Delta}\) is determined by a partition of \(6\).
Therefore, there are $11$ possibilities.
We now exclude the impossible cases.

By Lemma~\ref{lem:graph-closure} (2), if \(G_\Delta\) contains a \(2\)-matching, then it also contains a \(3\)-matching.
Hence any graph whose maximum matching number is exactly \(2\) cannot occur.
Therefore
\[
2K_2,\ K_2\sqcup K_3,\ 2K_3,\ K_4,\ K_5
\]
cannot occur.
This implies that the only remaining possibilities, up to graph isomorphism, are
\[
\emptyset,\ K_2,\ K_3,\ 3K_2,\ K_4\sqcup K_2,\ K_6.
\]
\end{proof}

We now obtain the classification theorem.

\begin{Theorem}\label{thm:degree3-classification}
Let \(\Delta\) be a \(5\)-dimensional Gorenstein simplex of degree \(3\) which is not a lattice pyramid.
Then, up to unimodular equivalence, \(\Delta\) is one of the six simplices listed in Table~\ref{tab:deg3}, where \(\mathbf{0}\) denotes the origin of \(\RR^5\) and \(\eb_1,\ldots,\eb_5\) denote the standard unit coordinate vectors.
\end{Theorem}
\begin{table}[h]
\centering
\caption{The \(5\)-dimensional Gorenstein simplices of degree \(3\) which are not lattice pyramids}
\label{tab:deg3}
\begin{tabular}{c|c|c|l}
type & graph \(G_\Delta\) & \(h^*(\Delta,t)\) &  vertices \\ \hline
\(\Delta^{(3)}_1\)
& \(\emptyset\)
& \(1+t^3\)
& \(\mathbf{0},\eb_1,\eb_2,\eb_3,\eb_4,\eb_1+\eb_2+\eb_3+\eb_4+2\eb_5\) \\ \hline

\(\Delta^{(3)}_2\)
& \(K_2\)
& \(1+t+t^2+t^3\)
& \(\mathbf{0},\eb_1,\eb_2,\eb_3,2\eb_4,\eb_1+\eb_2+\eb_3+2\eb_5\) \\ \hline

\(\Delta^{(3)}_3\)
& \(K_3\)
& \(1+3t+3t^2+t^3\)
& \(\mathbf{0},\eb_1,\eb_2,2\eb_3,2\eb_4,\eb_1+\eb_2+2\eb_5\) \\ \hline

\(\Delta^{(3)}_{4}\)
& \(3K_2\)
& \(1+3t+3t^2+t^3\)
& \(\mathbf{0},\eb_1,\eb_2,2\eb_3,\eb_2+2\eb_4,\eb_1+2\eb_5\) \\ \hline

\(\Delta^{(3)}_5\)
& \(K_4\sqcup K_2\)
& \(1+7t+7t^2+t^3\)
& \(\mathbf{0},\eb_1,2\eb_2,2\eb_3,2\eb_4,\eb_1+2\eb_5\) \\ \hline

\(\Delta^{(3)}_6\)
& \(K_6\)
& \(1+15t+15t^2+t^3\)
& \(\mathbf{0},2\eb_1,2\eb_2,2\eb_3,2\eb_4,2\eb_5\)
\end{tabular}
\end{table}

\begin{proof}
By Lemma~\ref{lem:graph-classification}, the graph \(G_\Delta\) is isomorphic to one of the following six graphs:
\[
\emptyset,\ K_2,\ K_3,\ 3K_2,\ K_4\sqcup K_2,\ K_6.
\]
For each of these graphs, the set \(H_1\) is determined, and hence so is the associated finite abelian group \(\Lambda_\Delta\), since \(\Lambda_\Delta\) is completely determined by \(H_1\).
Using Proposition~\ref{prop:delta-lambda-correspondence}, we obtain a representative simplex for each case, listed in Table~\ref{tab:deg3}.
The corresponding \(h^*\)-polynomials are computed from Corollary~\ref{cor:weight} by counting codewords of each weight, or equivalently from Lemma~\ref{lem:hstar-lambda}.
This gives exactly the six simplices listed in the table.
\end{proof}

\subsection{Classification of the $7$-dimensional Gorenstein simplices of degree $4$}\label{sec:degree4}

In this section, we consider \(7\)-dimensional Gorenstein simplices of degree \(4\) which are not lattice pyramids.
By Proposition~\ref{prop:H_i-determine}, the group \(\Lambda_\Delta\) is completely determined by \(H_1\) and \(H_2\).

First, we classify $G_{\Delta}$.
\begin{Proposition}\label{prop:H1-partition}
The graph \(G_\Delta\) is a disjoint union of complete graphs.
Moreover, up to graph isomorphism, \(G_\Delta\) is one of the following:
\[
\emptyset,\ 
K_2,\ 
2K_2,\ 
4K_2,\ 
K_3,\ 
K_3\sqcup K_2,\ 
2K_3,\ 
K_4,\ 
K_4 \sqcup 2K_2,\ 
2K_4,\ 
K_5,\ 
K_6\sqcup K_2,\ 
K_8.
\]
\end{Proposition}

\begin{proof}
By Lemma~\ref{lem:graph-closure} (1), every connected component of \(G_\Delta\) is a complete graph.
Hence \(G_\Delta\) is a disjoint union of complete graphs.
Thus, up to graph isomorphism, \(G_\Delta\) is determined by a partition of \(8\).
Therefore, there are $22$ possibilities.
We now exclude the impossible cases.

By Lemma~\ref{lem:graph-closure} (2), if \(G_\Delta\) contains a \(3\)-matching, then it also contains a \(4\)-matching.
Hence any graph whose maximum matching number is exactly \(3\) cannot occur.
Therefore
\[
3K_2,\ K_3\sqcup 2K_2,\ 2K_3\sqcup K_2,\ K_4\sqcup K_2,\ K_4\sqcup K_3,\ K_5\sqcup K_2,\ K_5\sqcup K_3,\ K_6,\ K_7
\]
cannot occur.

This implies that the only remaining possibilities, up to graph isomorphism, are
\[
\emptyset,\ 
K_2,\ 
2K_2,\ 
4K_2,\ 
K_3,\ 
K_3\sqcup K_2,\ 
2K_3,\ 
K_4,\ 
K_4 \sqcup 2K_2,\ 
2K_4,\ 
K_5,\ 
K_6\sqcup K_2,\ 
K_8.
\]
\end{proof}

Next, we consider \(H_2\).
For each \(A\subset \{1,\dots,8\}\) with \(|A|=4\), let \(\eta_A\in [0,1)^8\) be the element whose \(i\)-th coordinate is \(1/2\) if \(i\in A\), and \(0\) otherwise. Then every element of \(H_2\) is of the form \(\eta_A\). Hence \(H_2\) can be regarded as a \(4\)-uniform hypergraph on the vertex set \(\{1,\dots,8\}\).
For subsets \(A,B\subset \{1,\dots,8\}\), let
\[
A\triangle B:=(A\setminus B)\cup (B\setminus A)
\]
denote their symmetric difference.
The following lemma follows directly from the group structure of \(\Lambda_\Delta\) and the description of its elements in terms of supports.
\begin{Lemma}\label{lem:H1H2-hypergraph-closure}
The pair \((H_1,H_2)\) satisfies the following properties.
\begin{enumerate}
\item If \(\eta_A\in H_2\), then \(\eta_{A^c}\in H_2\), where \(A^c:=\{1,\dots,8\}\setminus A\).
\item If \(\varepsilon_{ij},\varepsilon_{k\ell}\in H_1\) and \(\{i,j\}\cap \{k,\ell\}=\emptyset\), then \(\eta_{\{i,j,k,\ell\}}\in H_2\).
\item Let \(A\subset \{1,\dots,8\}\) with \(|A|=4\), and assume that \(\eta_A\in H_2\).
\begin{enumerate}
\item If \(\varepsilon_{ij}\in H_1\) and \(|A\cap \{i,j\}|=1\), then \(\eta_{A\triangle \{i,j\}}\in H_2\).
\item If \(\varepsilon_{ij}\in H_1\) and \(\{i,j\}\subset A\), then \(\varepsilon_{A\setminus \{i,j\}}\in H_1\).
\item If \(\varepsilon_{ij}\in H_1\) and \(A\cap \{i,j\}=\emptyset\), then \(\varepsilon_{\{1,\dots,8\}\setminus (A\cup \{i,j\})}\in H_1\).
\end{enumerate}
\item Let \(A,B\subset \{1,\dots,8\}\) with \(|A|=|B|=4\), and assume that \(\eta_A,\eta_B\in H_2\).
\begin{enumerate}
\item If \(|A\cap B|=3\), then \(\varepsilon_{A\triangle B}\in H_1\).
\item If \(|A\cap B|=2\), then \(\eta_{A\triangle B}\in H_2\).
\item If \(|A\cap B|=1\), then \(\varepsilon_{\{1,\dots,8\}\setminus (A\triangle B)}\in H_1\).
\end{enumerate}
\end{enumerate}
\end{Lemma}

For each graph \(G_\Delta\) listed in Proposition~\ref{prop:H1-partition}, we determine all possible sets \(H_2\), and hence the associated finite abelian group \(\Lambda_\Delta\).

\begin{Proposition}[Case \(G_\Delta\cong \emptyset\)]\label{prop:case-H1-empty}
Assume that \(H_1=\emptyset\). Then, up to relabeling of the coordinates, \(H_2\) is one of the following four sets:
\begin{enumerate}
\item \(H_2=\emptyset\);
\item \(H_2=\{\eta_{\{1,2,3,4\}},\eta_{\{5,6,7,8\}}\}\);
\item \(H_2=\{\eta_{\{1,2,3,4\}},\eta_{\{1,2,5,6\}},\eta_{\{3,4,5,6\}},\eta_{\{1,2,7,8\}},\eta_{\{3,4,7,8\}},\eta_{\{5,6,7,8\}}\}\);
\item
\(H_2=\{\eta_A:A\in \mathcal B_0\}\), where
\[
\begin{aligned}
\mathcal B_0:=\{&
\{1,2,3,4\},\{1,2,5,6\},\{1,2,7,8\},\{1,3,5,7\},\{1,3,6,8\},\{1,4,5,8\},\{1,4,6,7\},\\
&
\{5,6,7,8\},\{3,4,7,8\},\{3,4,5,6\},\{2,4,6,8\},\{2,4,5,7\},\{2,3,6,7\},\{2,3,5,8\}
\}.
\end{aligned}
\]
\end{enumerate}
\end{Proposition}

\begin{proof}
Since \(H_1=\emptyset\), Lemma~\ref{lem:H1H2-hypergraph-closure} (4a) and (4c) imply that if \(\eta_A,\eta_B\in H_2\), then \(|A\cap B|\neq 1,3\). Thus, for distinct \(A,B\) with \(\eta_A,\eta_B\in H_2\), one has \(|A\cap B|=0\) or \(2\). If \(H_2=\emptyset\), we are in case (1).

Assume that \(H_2\neq \emptyset\), and choose \(A\subset \{1,\dots,8\}\) with \(\eta_A\in H_2\). Up to relabeling, we may assume that \(A=\{1,2,3,4\}\). Then \(\eta_{A^c}\in H_2\), where \(A^c=\{5,6,7,8\}\). If these are the only two elements of \(H_2\), then we are in case (2).

Now assume that there exists \(B\neq A,A^c\) with \(\eta_B\in H_2\). Since \(|A\cap B|=2\), after relabeling preserving \(A\), we may assume that \(B=\{1,2,5,6\}\). By Lemma~\ref{lem:H1H2-hypergraph-closure} (4b), one has \(\eta_{\{3,4,5,6\}}\in H_2\). Taking complements, we also obtain \(\eta_{\{1,2,7,8\}},\eta_{\{3,4,7,8\}}\in H_2\). Thus \(H_2\) contains at least the six elements listed in case (3). If these are all the elements of \(H_2\), then we are in case (3).

Finally, assume that there exists another \(D\subset \{1,\dots,8\}\) with \(\eta_D\in H_2\), different from the above six sets. Since \(|A\cap D|=2\) and \(|B\cap D|=2\), after relabeling preserving \(A\) and \(B\), we may assume that \(D=\{1,3,5,7\}\). Applying Lemma~\ref{lem:H1H2-hypergraph-closure} (4b) repeatedly, we obtain \(\eta_{\{2,4,5,7\}},\eta_{\{2,3,6,7\}},\eta_{\{1,4,6,7\}}\in H_2\). Taking complements, we also obtain \(\eta_{\{1,3,6,8\}},\eta_{\{1,4,5,8\}},\eta_{\{2,4,6,8\}},\eta_{\{2,3,5,8\}}\in H_2\). Hence \(H_2\) contains all fourteen sets listed in case (4).

Let \(E\subset \{1,\dots,8\}\) with \(\eta_E\in H_2\). Since \(|E\cap A|,|E\cap B|,|E\cap D|\in \{0,2\}\), a direct check shows that \(E\) is one of the fourteen sets in case (4). Therefore we are in case (4).
\end{proof}

\begin{Proposition}[Case \(G_\Delta\cong K_2\)]\label{prop:case-H1-K2}
Assume that \(H_1=\{\varepsilon_{12}\}\). Then, up to relabeling of the coordinates \(3,4,5,6,7,8\), \(H_2\) is one of the following two possibilities:
\begin{enumerate}
\item \(H_2=\emptyset\);
\item \(H_2=\{\eta_{\{1,3,4,5\}},\eta_{\{2,3,4,5\}},\eta_{\{1,6,7,8\}},\eta_{\{2,6,7,8\}}\}\).
\end{enumerate}
\end{Proposition}

\begin{proof}
Assume that \(H_2\neq \emptyset\), and let \(\eta_A\in H_2\). We first show that \(|A\cap \{1,2\}|=1\). Indeed, if \(\{1,2\}\subset A\), then Lemma~\ref{lem:H1H2-hypergraph-closure} (3b) yields \(\varepsilon_{A\setminus \{1,2\}}\in H_1\), which is impossible since \(A\setminus \{1,2\}\subset \{3,\dots,8\}\). If \(A\cap \{1,2\}=\emptyset\), then Lemma~\ref{lem:H1H2-hypergraph-closure} (3c) yields \(\varepsilon_{\{1,\dots,8\}\setminus (A\cup \{1,2\})}\in H_1\), again impossible since its support is contained in \(\{3,\dots,8\}\). Thus every \(A\) with \(\eta_A\in H_2\) satisfies \(|A\cap \{1,2\}|=1\).

Now let \(\mathcal T:=\{\,B\subset \{3,4,5,6,7,8\}: |B|=3,\ \eta_{\{1\}\cup B}\in H_2\,\}\). Then \(H_2\) is completely determined by \(\mathcal T\), since Lemma~\ref{lem:H1H2-hypergraph-closure} (3a) implies that \(\eta_{\{1\}\cup B}\in H_2\) forces \(\eta_{\{2\}\cup B}\in H_2\), and taking complements yields \(\eta_{\{1\}\cup B^c},\eta_{\{2\}\cup B^c}\in H_2\), where \(B^c:=\{3,4,5,6,7,8\}\setminus B\). In particular, \(\mathcal T\) is closed under taking complements.

Let \(B,C\in \mathcal T\) be distinct. If \(|B\cap C|=2\), then \(|(\{1\}\cup B)\cap(\{1\}\cup C)|=3\), so Lemma~\ref{lem:H1H2-hypergraph-closure} (4a) yields \(\varepsilon_{B\triangle C}\in H_1\), impossible since \(B\triangle C\subset \{3,\dots,8\}\). If \(|B\cap C|=1\), then \(|(\{1\}\cup B)\cap(\{1\}\cup C)|=2\), so Lemma~\ref{lem:H1H2-hypergraph-closure} (4b) yields \(\eta_{B\triangle C}\in H_2\), impossible because \(B\triangle C\subset \{3,\dots,8\}\) and hence meets \(\{1,2\}\) in no points. Thus \(B\cap C=\emptyset\), so \(C=B^c\). Therefore \(\mathcal T\) is either empty or consists of one complementary pair \(\{B,B^c\}\).

If \(\mathcal T=\emptyset\), then \(H_2=\emptyset\), giving case (1). If \(\mathcal T=\{B,B^c\}\), then, up to relabeling of the coordinates \(3,4,5,6,7,8\), we may assume that \(B=\{3,4,5\}\) and \(B^c=\{6,7,8\}\). Hence \[H_2=\{\eta_{\{1,3,4,5\}},\eta_{\{2,3,4,5\}},\eta_{\{1,6,7,8\}},\eta_{\{2,6,7,8\}}\},\] which is case (2).
\end{proof}

\begin{Proposition}[Case \(G_\Delta\cong 2K_2\)]\label{prop:case-H1-2K2}
Assume that \(H_1=\{\varepsilon_{12},\varepsilon_{34}\}\). Then, up to relabeling of the coordinates \(5,6,7,8\), \(H_2\) is one of the following two possibilities:
\begin{enumerate}
\item \(H_2=\{\eta_{\{1,2,3,4\}},\eta_{\{5,6,7,8\}}\}\);
\item \(H_2=\{\eta_A:A\in \mathcal C_0\}\), where
\[
\begin{aligned}
\mathcal C_0:=\{&
\{1,2,3,4\},\{5,6,7,8\},\{1,3,5,6\},\{2,3,5,6\},\{1,4,5,6\},\{2,4,5,6\},\\
&
\{1,3,7,8\},\{2,3,7,8\},\{1,4,7,8\},\{2,4,7,8\}
\}.
\end{aligned}
\]
\end{enumerate}
\end{Proposition}

\begin{proof}
By Lemma~\ref{lem:H1H2-hypergraph-closure} (2) and (1), one has \(\eta_{\{1,2,3,4\}},\eta_{\{5,6,7,8\}}\in H_2\).

Let \(\eta_A\in H_2\). If \(\{1,2\}\subset A\), then Lemma~\ref{lem:H1H2-hypergraph-closure} (3b) yields \(\varepsilon_{A\setminus \{1,2\}}\in H_1\), so \(A\setminus \{1,2\}=\{3,4\}\), that is, \(A=\{1,2,3,4\}\). Similarly, if \(\{3,4\}\subset A\), then \(A=\{1,2,3,4\}\). If \(A\cap \{1,2\}=\emptyset\), then Lemma~\ref{lem:H1H2-hypergraph-closure} (3c) yields \(\varepsilon_{\{1,\dots,8\}\setminus (A\cup \{1,2\})}\in H_1\), so \(\{1,\dots,8\}\setminus (A\cup \{1,2\})=\{3,4\}\), hence \(A=\{5,6,7,8\}\). Similarly, if \(A\cap \{3,4\}=\emptyset\), then \(A=\{5,6,7,8\}\). Therefore every other element \(\eta_A\in H_2\) satisfies \(|A\cap \{1,2\}|=|A\cap \{3,4\}|=1\).

Suppose that such an additional element exists. After relabeling \(1\leftrightarrow 2\), \(3\leftrightarrow 4\), and \(5,6,7,8\), we may assume that \(\eta_{\{1,3,5,6\}}\in H_2\). Applying Lemma~\ref{lem:H1H2-hypergraph-closure} (3a) to \(\varepsilon_{12}\) and \(\varepsilon_{34}\), we obtain \[\eta_{\{2,3,5,6\}},\eta_{\{1,4,5,6\}},\eta_{\{2,4,5,6\}}\in H_2.\] Taking complements gives \(\eta_{\{2,4,7,8\}},\eta_{\{1,4,7,8\}},\eta_{\{2,3,7,8\}},\eta_{\{1,3,7,8\}}\in H_2\). Thus case (2) is forced.

It remains to show that no further element can occur. Let \(\eta_C\in H_2\). If \(C\notin \{\{1,2,3,4\},\{5,6,7,8\}\}\), then \(|C\cap \{1,2\}|=|C\cap \{3,4\}|=1\). Replacing \(C\) by \(C^c\) if necessary, we may write \(C=\{1,3\}\cup D\) for some \(D\subset \{5,6,7,8\}\) with \(|D|=2\). If \(D\neq \{5,6\},\{7,8\}\), then \(|D\cap \{5,6\}|=1\), so \(|(\{1,3\}\cup D)\cap \{1,3,5,6\}|=3\). Lemma~\ref{lem:H1H2-hypergraph-closure} (4a) then gives \(\varepsilon_{(\{1,3\}\cup D)\triangle \{1,3,5,6\}}\in H_1\), but its support is contained in \(\{5,6,7,8\}\), a contradiction. Hence \(D=\{5,6\}\) or \(\{7,8\}\), and \(C\) is one of the sets listed in case (2).
\end{proof}

\begin{Proposition}[Case \(G_\Delta\cong K_3\)]\label{prop:case-H1-K3}
Assume that \(H_1=\{\varepsilon_{12},\varepsilon_{13},\varepsilon_{23}\}\). Then \(H_2=\emptyset\).
\end{Proposition}

\begin{proof}
Suppose that \(\eta_A\in H_2\) for some \(A\subset \{1,\dots,8\}\) with \(|A|=4\). Since \(\eta_{A^c}\in H_2\), one has \(|A\cap \{1,2,3\}|+|A^c\cap \{1,2,3\}|=3\), so we may assume that \(|A\cap \{1,2,3\}|\geq 2\). Then \(A\) contains one of the edges \(\{1,2\},\{1,3\},\{2,3\}\), say \(e\). By Lemma~\ref{lem:H1H2-hypergraph-closure} (3b), one has \(\varepsilon_{A\setminus e}\in H_1\). However, \(A\setminus e\) contains one element of \(\{1,2,3\}\) and one element of \(\{4,5,6,7,8\}\), a contradiction.
\end{proof}

\begin{Proposition}[Case \(G_\Delta\cong 4 K_2\)]\label{prop:case-H1-4K2}
Assume that \(H_1=\{\varepsilon_{12},\varepsilon_{34},\varepsilon_{56},\varepsilon_{78}\}\). Let
\[
\mathcal A_0:=\{\{1,2,3,4\},\ \{1,2,5,6\},\ \{1,2,7,8\},\ \{3,4,5,6\},\ \{3,4,7,8\},\ \{5,6,7,8\}\}
\]
and
\[
\mathcal T:=\{A\subset \{1,\dots,8\}: |A|=4,\ |A\cap \{1,2\}|=|A\cap \{3,4\}|=|A\cap \{5,6\}|=|A\cap \{7,8\}|=1\}.
\]
Then \(H_2\) is one of the following two sets:
\begin{enumerate}
\item \(H_2=\{\eta_A:A\in \mathcal A_0\}\);
\item \(H_2=\{\eta_A:A\in \mathcal A_0\cup \mathcal T\}\).
\end{enumerate}
\end{Proposition}

\begin{proof}
By Lemma~\ref{lem:H1H2-hypergraph-closure} (2), one has \(\{\eta_A:A\in \mathcal A_0\}\subset H_2\).

Let \(\eta_A\in H_2\). If \(A\) contains one of the pairs \(\{1,2\},\{3,4\},\{5,6\},\{7,8\}\), say \(\{1,2\}\subset A\), then Lemma~\ref{lem:H1H2-hypergraph-closure} (3b) yields \(\varepsilon_{A\setminus \{1,2\}}\in H_1\). Hence \(A\setminus \{1,2\}\) must be one of \(\{3,4\},\{5,6\},\{7,8\}\), so \(A\in \mathcal A_0\). If \(A\cap \{1,2\}=\emptyset\), then Lemma~\ref{lem:H1H2-hypergraph-closure} (3c) yields \(\varepsilon_{\{1,\dots,8\}\setminus (A\cup \{1,2\})}\in H_1\), so \(\{1,\dots,8\}\setminus (A\cup \{1,2\})\) is one of \(\{3,4\},\{5,6\},\{7,8\}\), and again \(A\in \mathcal A_0\). Thus if \(A\notin \mathcal A_0\), then \(A\) meets each of the four pairs \(\{1,2\},\{3,4\},\{5,6\},\{7,8\}\) in exactly one point, so \(A\in \mathcal T\).

Now suppose that there exists \(A\in \mathcal T\) with \(\eta_A\in H_2\). By relabeling inside each pair, we may assume that \(A=\{1,3,5,7\}\). Applying Lemma~\ref{lem:H1H2-hypergraph-closure} (3a) repeatedly to \(\varepsilon_{12},\varepsilon_{34},\varepsilon_{56},\varepsilon_{78}\), we obtain that every set in \(\mathcal T\) also belongs to \(H_2\). Hence \(\{\eta_A:A\in \mathcal A_0\cup \mathcal T\}\subset H_2\). Since every element of \(H_2\) lies in \(\mathcal A_0\cup \mathcal T\), the conclusion follows.
\end{proof}

\begin{Proposition}[Case \(G_\Delta\cong K_3 \sqcup K_2\)]\label{prop:case-H1-K3K2}
Assume that \(H_1=\{\varepsilon_{12},\varepsilon_{13},\varepsilon_{23},\varepsilon_{45}\}\). Then
\[
H_2=\{\eta_{\{1,2,4,5\}},\eta_{\{1,3,4,5\}},\eta_{\{2,3,4,5\}},\eta_{\{3,6,7,8\}},\eta_{\{2,6,7,8\}},\eta_{\{1,6,7,8\}}\}.
\]
\end{Proposition}

\begin{proof}
By Lemma~\ref{lem:H1H2-hypergraph-closure}, \(H_2\) contains all six sets listed in the statement.

Let \(\eta_A\in H_2\). If \(\{4,5\}\subset A\), then Lemma~\ref{lem:H1H2-hypergraph-closure} (3b) yields \(\varepsilon_{A\setminus \{4,5\}}\in H_1\), so \(A\setminus \{4,5\}\) must be one of \(\{1,2\},\{1,3\},\{2,3\}\). Hence \(A\) is one of \(\{1,2,4,5\},\{1,3,4,5\},\{2,3,4,5\}\). If \(A\cap \{4,5\}=\emptyset\), then Lemma~\ref{lem:H1H2-hypergraph-closure} (3c) yields \(\varepsilon_{\{1,\dots,8\}\setminus (A\cup \{4,5\})}\in H_1\), so this support must be one of \(\{1,2\},\{1,3\},\{2,3\}\). Hence \(A\) is one of \(\{3,6,7,8\},\{2,6,7,8\},\{1,6,7,8\}\).

Therefore, if \(A\) is not one of the above six sets, then \(|A\cap \{4,5\}|=1\). 
If \(A\) contains one of the edges \(\{1,2\},\{1,3\},\{2,3\}\), say $e$, then it then follows from Lemma~\ref{lem:H1H2-hypergraph-closure} (3b) that $\varepsilon_{A \setminus e} \in H_1$. However, $A \setminus e$ contains one element of $\{4,5\}$ and one element of $\{1,2,3,6,7,8\}$, a contradiction. Hence \(A\) cannot contain any of the edges \(\{1,2\},\{1,3\},\{2,3\}\). This implies  \(|A\cap \{1,2,3\}|\le 1\). If \(|A\cap \{1,2,3\}|=0\), then \(A\cap \{1,2\}=\emptyset\), and Lemma~\ref{lem:H1H2-hypergraph-closure} (3c) gives \(\varepsilon_{\{1,\dots,8\}\setminus (A\cup \{1,2\})}\in H_1\); but this support is of the form \(\{3,i\}\) with \(i\in \{4,5\}\), impossible. If \(|A\cap \{1,2,3\}|=1\), let \(i\) be the unique element of \(A\cap \{1,2,3\}\), and let \(\{j,k\}\) be the complementary edge in \(\{1,2,3\}\). Then \(A\cap \{j,k\}=\emptyset\), so Lemma~\ref{lem:H1H2-hypergraph-closure} (3c) yields \(\varepsilon_{\{1,\dots,8\}\setminus (A\cup \{j,k\})}\in H_1\), but this support contains one element of \(\{4,5\}\) and one element of \(\{6,7,8\}\), again impossible. Hence no further element can occur.
\end{proof}

\begin{Proposition}[Case \(G_\Delta\cong 2K_3\)]\label{prop:case-H1-2K3}
Assume that \(H_1=\{\varepsilon_{12},\varepsilon_{13},\varepsilon_{23},\varepsilon_{45},\varepsilon_{46},\varepsilon_{56}\}\). Set \(T_1:=\{1,2,3\}\), \(T_2:=\{4,5,6\}\), \(\mathcal E_1:=\{E\subset T_1:|E|=2\}\), and \(\mathcal E_2:=\{F\subset T_2:|F|=2\}\). Then
\[
H_2=\{\eta_{E\cup F}:E\in \mathcal E_1,\ F\in \mathcal E_2\}\cup \{\eta_{(\{1,\ldots,8\}\setminus (E\cup F))}:E\in \mathcal E_1,\ F\in \mathcal E_2\}.
\]
\end{Proposition}

\begin{proof}
Since every edge in \(T_1\) is disjoint from every edge in \(T_2\), Lemma~\ref{lem:H1H2-hypergraph-closure} (2) implies that \(\eta_{E\cup F}\in H_2\) for every \(E\in \mathcal E_1\) and \(F\in \mathcal E_2\). By Lemma~\ref{lem:H1H2-hypergraph-closure} (1), also \(\eta_{(\{1,\ldots,8\}\setminus (E\cup F))}\in H_2\). Thus the right-hand side is contained in \(H_2\).

Conversely, let \(\eta_A\in H_2\). Since \(\eta_{A^c}\in H_2\), one has \(|A\cap T_1|+|A^c\cap T_1|=3\). Hence, replacing \(A\) by \(A^c\) if necessary, we may assume that \(|A\cap T_1|\geq 2\). If \(|A\cap T_1|=3\), then \(A\) contains an edge \(E\subset T_1\), and Lemma~\ref{lem:H1H2-hypergraph-closure} (3b) yields \(\varepsilon_{A\setminus E}\in H_1\). However, \(A\setminus E\) contains one element of \(T_1\) and one element of \(\{4,5,6,7,8\}\), impossible. Thus \(|A\cap T_1|=2\). Similarly, \(|A\cap T_2|\neq 0,1,3\), so \(|A\cap T_2|=2\). Since \(|A|=4\), it follows that \(A=E\cup F\) for some \(E\in \mathcal E_1\) and \(F\in \mathcal E_2\). Therefore every element of \(H_2\) is listed in the statement.
\end{proof}

\begin{Proposition}[Case \(G_\Delta\cong K_4\)]\label{prop:case-H1-K4}
Assume that \(H_1=\{\varepsilon_{ij}:1\le i<j\le 4\}\). Then \[H_2=\{\eta_{\{1,2,3,4\}},\eta_{\{5,6,7,8\}}\}.\]
\end{Proposition}

\begin{proof}
Since \(\varepsilon_{12}\) and \(\varepsilon_{34}\) are disjoint elements of \(H_1\), Lemma~\ref{lem:H1H2-hypergraph-closure} (2) yields \(\eta_{\{1,2,3,4\}}\in H_2\), and Lemma~\ref{lem:H1H2-hypergraph-closure} (1) yields \(\eta_{\{5,6,7,8\}}\in H_2\).

Conversely, let \(\eta_A\in H_2\), and set \(T:=\{1,2,3,4\}\). Since \(\eta_{A^c}\in H_2\), one has \(|A\cap T|+|A^c\cap T|=4\). Replacing \(A\) by \(A^c\) if necessary, we may assume that \(|A\cap T|\geq 2\). If \(|A\cap T|=4\), then \(A=T\). If \(|A\cap T|=3\), then \(A\) contains an edge \(E\subset T\), and Lemma~\ref{lem:H1H2-hypergraph-closure} (3b) yields \(\varepsilon_{A\setminus E}\in H_1\), impossible since \(A\setminus E\) contains one element of \(T\) and one element of \(T^c\). If \(|A\cap T|=2\), then again \(A\) contains an edge \(E\subset T\), and Lemma~\ref{lem:H1H2-hypergraph-closure} (3b) yields \(\varepsilon_{A\setminus E}\in H_1\), impossible since \(A\setminus E\subset T^c\). Hence only \(|A\cap T|=4\) is possible, and thus \(A=T\) or \(A=T^c\).
\end{proof}

\begin{Proposition}[Case \(G_\Delta\cong K_4 \sqcup 2K_2\)]\label{prop:case-H1-K4-2K2}
Assume that \(H_1=\{\varepsilon_{ij}:1\le i<j\le 4\}\cup \{\varepsilon_{56},\varepsilon_{78}\}\). Then
\[
\begin{aligned}
H_2
=
&\{\eta_{\{1,2,3,4\}},\eta_{\{5,6,7,8\}}\}\\
&\cup \{\eta_{E\cup \{5,6\}}:E\subset \{1,2,3,4\},\ |E|=2\}\\
&\cup \{\eta_{E\cup \{7,8\}}:E\subset \{1,2,3,4\},\ |E|=2\}.
\end{aligned}
\]
\end{Proposition}

\begin{proof}
By Lemma~\ref{lem:H1H2-hypergraph-closure} (2) and (1), the right-hand side is contained in \(H_2\).

Conversely, let \(\eta_A\in H_2\), and set \(T:=\{1,2,3,4\}\). Since \(\eta_{A^c}\in H_2\), one has \(|A\cap T|+|A^c\cap T|=4\). Replacing \(A\) by \(A^c\) if necessary, we may assume that \(|A\cap T|\geq 2\). If \(|A\cap T|=4\), then \(A=T\). If \(|A\cap T|=3\), then \(A\) contains an edge \(E\subset T\), and Lemma~\ref{lem:H1H2-hypergraph-closure} (3b) yields \(\varepsilon_{A\setminus E}\in H_1\), impossible since \(A\setminus E\) contains one element of \(T\) and one element of \(\{5,6,7,8\}\). Thus \(|A\cap T|=2\). Let \(E:=A\cap T\). Since \(E\) contains an edge of \(T\), Lemma~\ref{lem:H1H2-hypergraph-closure} (3b) gives \(\varepsilon_{A\setminus E}\in H_1\). As \(A\setminus E\subset \{5,6,7,8\}\), one has \(A\setminus E=\{5,6\}\) or \(\{7,8\}\). Hence \(A\) is one of the sets listed in the statement.
\end{proof}

\begin{Proposition}[Case \(G_\Delta\cong 2K_4\)]\label{prop:case-H1-2K4}
Assume that \(H_1=\{\varepsilon_{ij}:1\le i<j\le 4\}\cup \{\varepsilon_{ij}:5\le i<j\le 8\}\). Then
\[
H_2=\{\eta_A: A\subset \{1,\dots,8\},\ |A|=4,\ |A\cap \{1,2,3,4\}|\in \{0,2,4\}\}.
\]
\end{Proposition}

\begin{proof}
By Lemma~\ref{lem:H1H2-hypergraph-closure} (2) and (1), the right-hand side is contained in \(H_2\).

Conversely, let \(\eta_A\in H_2\), and set \(T:=\{1,2,3,4\}\). Since \(\eta_{A^c}\in H_2\), one has \(|A\cap T|+|A^c\cap T|=4\). Replacing \(A\) by \(A^c\) if necessary, we may assume that \(|A\cap T|\geq 2\). If \(|A\cap T|=4\), then we are done. If \(|A\cap T|=3\), then \(A\) contains an edge \(E\subset T\), and Lemma~\ref{lem:H1H2-hypergraph-closure} (3b) yields \(\varepsilon_{A\setminus E}\in H_1\), impossible since \(A\setminus E\) contains one element of \(T\) and one element of \(T^c\). Thus \(|A\cap T|=2\), which is allowed. Therefore \(|A\cap T|\in \{0,2,4\}\), as desired.
\end{proof}

\begin{Proposition}[Case \(G_\Delta\cong K_5\)]\label{prop:case-H1-K5}
Assume that \(H_1=\{\varepsilon_{ij}:1\le i<j\le 5\}\). Then
\[
H_2=\{\eta_A:A\subset \{1,2,3,4,5\},\ |A|=4\}\cup \{\eta_{\{i,6,7,8\}}:1\le i\le 5\}.
\]
\end{Proposition}

\begin{proof}
By Lemma~\ref{lem:H1H2-hypergraph-closure} (2) and (1), the right-hand side is contained in \(H_2\).

Conversely, let \(\eta_A\in H_2\), and set \(T:=\{1,2,3,4,5\}\). Since \(|A|=4\), one has \(1\le |A\cap T|\le 4\). If \(|A\cap T|=4\), then \(A\subset T\), so \(A\) is listed. If \(|A\cap T|=3\), then \(A\) contains an edge \(E\subset T\), and Lemma~\ref{lem:H1H2-hypergraph-closure} (3b) yields \(\varepsilon_{A\setminus E}\in H_1\), impossible since \(A\setminus E\) contains one element of \(T\) and one element of \(\{6,7,8\}\). If \(|A\cap T|=2\), then again \(A\) contains an edge \(E\subset T\), and Lemma~\ref{lem:H1H2-hypergraph-closure} (3b) yields \(\varepsilon_{A\setminus E}\in H_1\), impossible since \(A\setminus E\subset \{6,7,8\}\). Hence \(|A\cap T|=1\), and therefore \(A=\{i,6,7,8\}\) for some \(1\le i\le 5\).
\end{proof}

\begin{Proposition}[Case \(G_\Delta\cong K_6 \sqcup K_2\)]\label{prop:case-H1-K6K2}
Assume that \(H_1=\{\varepsilon_{ij}:1\le i<j\le 6\}\cup \{\varepsilon_{78}\}\). Then
\[
H_2=\{\eta_A:A\subset \{1,\dots,6\},\ |A|=4\}\cup \{\eta_{E\cup \{7,8\}}:E\subset \{1,\dots,6\},\ |E|=2\}.
\]
\end{Proposition}

\begin{proof}
By Lemma~\ref{lem:H1H2-hypergraph-closure} (2) and (1), the right-hand side is contained in \(H_2\).

Conversely, let \(\eta_A\in H_2\), and set \(T:=\{1,\dots,6\}\). Since \(|\{7,8\}|=2\), one has \(2\le |A\cap T|\le 4\). If \(|A\cap T|=4\), then \(A\subset T\), so \(A\) is listed. If \(|A\cap T|=3\), then \(A\) contains an edge \(E\subset T\), and Lemma~\ref{lem:H1H2-hypergraph-closure} (3b) yields \(\varepsilon_{A\setminus E}\in H_1\), impossible since \(A\setminus E\) contains one element of \(T\) and one element of \(\{7,8\}\). Thus \(|A\cap T|=2\). Let \(E:=A\cap T\). Since \(E\) is an edge of \(H_1\), Lemma~\ref{lem:H1H2-hypergraph-closure} (3b) yields \(\varepsilon_{A\setminus E}\in H_1\). As \(A\setminus E\subset \{7,8\}\), one has \(A\setminus E=\{7,8\}\). Hence \(A=E\cup \{7,8\}\), as desired.
\end{proof}

\begin{Proposition}[Case \(G_\Delta\cong K_8\)]\label{prop:case-H1-K8}
Assume that \(H_1=\{\varepsilon_{ij}:1\le i<j\le 8\}\). Then \(H_2=\{\eta_A:A\subset \{1,\dots,8\},\ |A|=4\}\).
\end{Proposition}

\begin{proof}
Let \(A\subset \{1,\dots,8\}\) with \(|A|=4\). Writing \(A=\{i,j\}\sqcup \{k,\ell\}\), one has \(\varepsilon_{ij},\varepsilon_{k\ell}\in H_1\), so Lemma~\ref{lem:H1H2-hypergraph-closure} (2) yields \(\eta_A\in H_2\). Hence every \(4\)-subset of \(\{1,\dots,8\}\) gives an element of \(H_2\).
\end{proof}

We now obtain the classification theorem.

\begin{Theorem}\label{thm:degree4-classification}
Let \(\Delta\) be a \(7\)-dimensional Gorenstein simplex of degree \(4\) which is not a lattice pyramid.
Then, up to unimodular equivalence, \(\Delta\) is one of the nineteen simplices listed in Table~\ref{tab:deg4}, where \(\mathbf{0}\) denotes the origin of \(\RR^7\) and \(\eb_1,\ldots,\eb_7\) denote the standard unit coordinate vectors.
\end{Theorem}

\begin{table}[h]
\centering
\caption{The \(7\)-dimensional Gorenstein simplices of degree \(4\) which are not lattice pyramids}
\label{tab:deg4}
\begin{tabular}{c|c|c|l}
type & graph \(G_\Delta\) & \(h^*(\Delta,t)\) & vertices \\ \hline
\(\Delta^{(4)}_1\)
& \(\emptyset\)
& \(1+t^4\)
&\begin{tabular}[t]{@{}l@{}}\(\mathbf{0},\,\eb_1,\,\eb_2,\,\eb_3,\,\eb_4, \, \eb_5,\,\eb_6,\)\\ \(\eb_1+\eb_2+\eb_3+\eb_4+\eb_5+\eb_6+2\eb_7\)
\end{tabular}
\\ \hline

\(\Delta^{(4)}_2\)
& \(\emptyset\)
& \(1+2t^2+t^4\)
& \begin{tabular}[t]{@{}l@{}}\(\mathbf{0},\,\eb_1,\,\eb_2,\,\eb_1+\eb_2+2\eb_3,\,\eb_4, \eb_5,\,\eb_6,\) \\ \(\eb_4+\eb_5+\eb_6+2\eb_7\)
\end{tabular}
\\ \hline

\(\Delta^{(4)}_3\)
& \(\emptyset\)
& \(1+6t^2+t^4\)
& \begin{tabular}[t]{@{}l@{}}\(\mathbf{0},\,\eb_1,\,\eb_2,\,\eb_3,\,\eb_4,\,
\eb_1+\eb_4+2\eb_5,\,
\eb_2+\eb_4+2\eb_6,\) \\ \(
\eb_3+\eb_4+2\eb_7\)
\end{tabular}
\\ \hline

\(\Delta^{(4)}_4\)
& \(\emptyset\)
& \(1+14t^2+t^4\)
& \begin{tabular}[t]{@{}l@{}}\(\mathbf{0},\,\eb_1,\,\eb_2,\,\eb_3,\,
\eb_2+\eb_3+2\eb_4,\,
\eb_1+\eb_3+2\eb_5,\)\\ \(
\eb_1+\eb_2+2\eb_6,\,
\eb_1+\eb_2+\eb_3+2\eb_7\)
\end{tabular}
\\ \hline

\(\Delta^{(4)}_5\)
& \(K_2\)
& \(1+t+t^3+t^4\)
& \begin{tabular}[t]{@{}l@{}}\(\mathbf{0},\,\eb_1,\,\eb_2,\,\eb_3,\,\eb_4,\,\eb_5,\,
\eb_1+2\eb_6,\) \\ \(
\eb_2+\eb_3+\eb_4+\eb_5+2\eb_7\)
\end{tabular}
\\ \hline

\(\Delta^{(4)}_6\)
& \(K_2\)
& \(1+t+4t^2+t^3+t^4\)
& \(\mathbf{0},\,\eb_1,\,\eb_2,\,\eb_3,\,\eb_4,\,
2\eb_5,\,
\eb_1+\eb_2+2\eb_6,\,
\eb_3+\eb_4+2\eb_7\)
\\ \hline

\(\Delta^{(4)}_7\)
& \(2K_2\)
& \(1+2t+2t^2+2t^3+t^4\)
& \(\mathbf{0},\,\eb_1,\,\eb_2,\,\eb_3,\,\eb_4,\,
\eb_1+2\eb_5,\,
\eb_2+2\eb_6,\,
\eb_3+\eb_4+2\eb_7\)
\\ \hline

\(\Delta^{(4)}_8\)
& \(2K_2\)
& \(1+2t+10t^2+2t^3+t^4\)
& \begin{tabular}[t]{@{}l@{}}\(\mathbf{0},\,\eb_1,\,\eb_2,\,\eb_3,\,
2\eb_4,\,
\eb_1+2\eb_5,\,
\eb_1+\eb_2+2\eb_6,\) \\ \(
\eb_1+\eb_3+2\eb_7\)
\end{tabular}
\\ \hline

\(\Delta^{(4)}_9\)
& \(K_3\)
& \(1+3t+3t^3+t^4\)
& \(\mathbf{0},\,\eb_1,\,\eb_2,\,\eb_3,\,\eb_4,\,
2\eb_5,\,
2\eb_6,\,
\eb_1+\eb_2+\eb_3+\eb_4+2\eb_7\)
\\ \hline

\(\Delta^{(4)}_{10}\)
& \(4K_2\)
& \(1+4t+6t^2+4t^3+t^4\)
& 
\(\mathbf{0},\,\eb_1,\,\eb_1+2\eb_2,\,\eb_3,\,\eb_3+2\eb_4,\eb_5,\,\eb_5+2\eb_6,\, 2\eb_7\)
\\ \hline

\(\Delta^{(4)}_{11}\)
& \(4K_2\)
& \(1+4t+22t^2+4t^3+t^4\)
& \begin{tabular}[t]{@{}l@{}}\(\mathbf{0},\,\eb_1,\,\eb_2,\,2\eb_3,\,\eb_2+2\eb_4,\,\eb_1+2\eb_5,\,\eb_1+\eb_2+2\eb_6,\)\\ \(\eb_1+\eb_2+2\eb_7\)
\end{tabular}
\\ \hline

\(\Delta^{(4)}_{12}\)
& \(K_3\sqcup K_2\)
& \(1+4t+6t^2+4t^3+t^4\)
& \(\mathbf{0},\,\eb_1,\,\eb_2,\,\eb_3,\,
2\eb_4,\,
2\eb_5,\,
\eb_1+2\eb_6,\,
\eb_2+\eb_3+2\eb_7\)
\\ \hline

\(\Delta^{(4)}_{13}\)
& \(2K_3\)
& \(1+6t+18t^2+6t^3+t^4\)
& \begin{tabular}[t]{@{}l@{}}\(\mathbf{0},\,\eb_1,\,\eb_2,\,
2\eb_3,\,
2\eb_4,\,
\eb_1+2\eb_5,\,
\eb_1+2\eb_6,\) \\ \(
\eb_1+\eb_2+2\eb_7\)
\end{tabular}
\\ \hline

\(\Delta^{(4)}_{14}\)
& \(K_4\)
& \(1+6t+2t^2+6t^3+t^4\)
& \(\mathbf{0},\,\eb_1,\,\eb_2,\,\eb_3,\,
2\eb_4,\,
2\eb_5,\,
2\eb_6,\,
\eb_1+\eb_2+\eb_3+2\eb_7\)
\\ \hline

\(\Delta^{(4)}_{15}\)
& \(K_4\sqcup 2K_2\)
& \(1+8t+14t^2+8t^3+t^4\)
& \(\mathbf{0},\,\eb_1,\,\eb_2,\,
2\eb_3,\,
2\eb_4,\,
2\eb_5,\,
\eb_1+2\eb_6,\,
\eb_2+2\eb_7\)
\\ \hline

\(\Delta^{(4)}_{16}\)
& \(2K_4\)
& \(1+12t+38t^2+12t^3+t^4\)
& \(\mathbf{0},\,\eb_1,\,
2\eb_2,\,
2\eb_3,\,
2\eb_4,\,
\eb_1+2\eb_5,\,
\eb_1+2\eb_6,\,
\eb_1+2\eb_7\)
\\ \hline

\(\Delta^{(4)}_{17}\)
& \(K_5\)
& \(1+10t+10t^2+10t^3+t^4\)
& \(\mathbf{0},\,\eb_1,\,\eb_2,\,
2\eb_3,\,
2\eb_4,\,
2\eb_5,\,
2\eb_6,\,
\eb_1+\eb_2+2\eb_7\)
\\ \hline

\(\Delta^{(4)}_{18}\)
& \(K_6\sqcup K_2\)
& \(1+16t+30t^2+16t^3+t^4\)
& \(\mathbf{0},\,\eb_1,\,
2\eb_2,\,
2\eb_3,\,
2\eb_4,\,
2\eb_5,\,
2\eb_6,\,
\eb_1+2\eb_7\)
\\ \hline

\(\Delta^{(4)}_{19}\)
& \(K_8\)
& \(1+28t+70t^2+28t^3+t^4\)
& \(\mathbf{0},\,2\eb_1,\,2\eb_2,\,2\eb_3,\,2\eb_4, \, 2\eb_5,\,2\eb_6,\,2\eb_7\)
\end{tabular}
\end{table}

\begin{proof}
By Proposition~\ref{prop:H_i-determine}, the group \(\Lambda_\Delta\) is completely determined by \(H_1\) and \(H_2\).
By Proposition~\ref{prop:H1-partition}, the possible graphs \(G_\Delta\) are exactly
\[
\emptyset,\ 
K_2,\ 
2K_2,\ 
4K_2,\ 
K_3,\ 
K_3\sqcup K_2,\ 
2K_3,\ 
K_4,\ 
K_4 \sqcup 2K_2,\ 
2K_4,\ 
K_5,\ 
K_6\sqcup K_2,\ 
K_8.
\]
For each of these graphs, Propositions~\ref{prop:case-H1-empty}, \ref{prop:case-H1-K2}, \ref{prop:case-H1-2K2}, \ref{prop:case-H1-K3}, \ref{prop:case-H1-4K2}, \ref{prop:case-H1-K3K2}, \ref{prop:case-H1-2K3}, \ref{prop:case-H1-K4}, \ref{prop:case-H1-K4-2K2}, \ref{prop:case-H1-2K4}, \ref{prop:case-H1-K5}, \ref{prop:case-H1-K6K2}, and \ref{prop:case-H1-K8} determine all possibilities for \(H_2\).
Using Proposition~\ref{prop:delta-lambda-correspondence}, we obtain a representative simplex for each case, listed in Table~\ref{tab:deg4}.
The corresponding \(h^*\)-polynomials are computed from Corollary~\ref{cor:weight} by counting codewords of each weight, or equivalently from Lemma~\ref{lem:hstar-lambda}.
This gives exactly the nineteen simplices listed in the table.
\end{proof}
\bibliographystyle{plain}
\bibliography{bibliography}
\end{document}